\documentclass{amsart}
\usepackage[utf8]{inputenc}    
\usepackage[english]{babel}     
\usepackage[a4paper, left=1.5in, right=1.5in, top=1.5in, bottom=1.5in]{geometry}
\usepackage{graphicx}          
\usepackage{amsmath,amssymb}   
\usepackage{amsthm,mathrsfs}           
\usepackage{mathtools}          
\usepackage{enumitem}           
\usepackage{listings}           
\usepackage{todonotes}          
\usepackage{hyperref}           
\usepackage{amscd}
\usepackage{color}
\usepackage[all]{xy}

\numberwithin{equation}{section}

\title{F-isocrystals of Higher Direct Images of $p$-Divisible Groups}
\author{Zhenghui Li, Yanshuai Qin}
\date{}

\newtheorem{theorem}{Theorem}[section]
\newtheorem{corollary}[theorem]{Corollary}

\newtheorem{lemma}[theorem]{Lemma}
\newtheorem{proposition}[theorem]{Proposition}
\theoremstyle{definition}
\newtheorem{definition}[theorem]{Definition}

\newtheorem{remark}[theorem]{Remark}

\def\bF{{\mathbb F}}
\def\bG{{\mathbb G}}

\def\bN{{\mathbb N}}

\def\bQ{{\mathbb Q}}

\def\bZ{{\mathbb Z}}

\def\cC{{\mathcal C}}

\def\cE{{\mathcal E}}
\def\cF{{\mathcal F}}

\def\cI{{\mathcal I}}
\def\cJ{{\mathcal J}}

\def\cM{{\mathcal M}}

\def\cO{{\mathcal O}}

\def\sC{{\mathscr{C}}}

\def\sF{{\mathscr{F}}}
\def\sG{{\mathscr{G}}}
\def\sH{{\mathscr{H}}}

\newcommand{\fppf}{\mathrm{fppf}} 

\newcommand{\syn}{\mathrm{syn}}

\newcommand{\lra}{{\longrightarrow}}

\newcommand{\SYN}{\mathrm{SYN}}

\def\#{{\sharp}}

\DeclareMathOperator{\colim}{colim}

\DeclareMathOperator{\Hom}{Hom}

\DeclareMathOperator{\Spec}{Spec}

\DeclareMathOperator{\im}{Im}
\DeclareMathOperator{\coim}{Coim}

\DeclareMathOperator{\coker}{coker}

\DeclareMathOperator{\Ext}{Ext}

\DeclareMathOperator{\Fisoc}{F-Isoc}

\DeclareMathOperator{\crys}{crys}
\DeclareMathOperator{\CRYS}{CRYS}
\DeclareMathOperator{\Crys}{Crys}

\DeclareMathOperator{\Fil}{Fil}

\DeclareMathOperator{\ZAR}{ZAR}
\DeclareMathOperator{\Div}{Div_p}

\begin{document}

\maketitle

\begin{abstract}
For a $p$-divisible group $G$ over a smooth projective variety $X \to \Spec (k)$, where $k$ is a field finitely generated over a perfect field of characteristic $p$, we show that the formal group $R^i f_{\fppf*} G$ is isogenous to a $p$-divisible group. The Dieudonné crystal of its divisible part is canonically isomorphic to the slope-$[0,1]$ part of $R^i f_{\crys*} \cM^{cr}(G)$ in the category of $F$-isocrystals over $\Spec(k)$. This provides an answer to the rational form of a question of Artin--Mazur regarding the enlarged formal Brauer groups.
\end{abstract}

\tableofcontents

\section{Introduction}
Let \( X \) be a smooth proper variety over a perfect field \( k \) of characteristic \( p > 0 \). Artin and Mazur \cite{ArtinMazurformal} defined a functor  
\[
\Phi^i \colon \mathrm{Art}/k \rightarrow \mathrm{Ab}
\]
by  
\[
A \mapsto \ker\left(H^i(X_A, \mathbb{G}_m) \rightarrow H^i(X, \mathbb{G}_m)\right),
\]
and computed its Cartier module as  
\[
D\Phi^i(X, \mathbb{G}_m) = H^i(X, W\mathcal{O}_X).
\]

In the case \( i = 2 \), the functor \( \Phi^2(X, \mathbb{G}_m) \) is known as the \emph{formal Brauer group} of \( X \). Artin and Mazur \cite[p. 106, Cor. 2.11]{ArtinMazurformal} showed that this functor is pro-representable when the Picard variety of \( X \) is smooth. The group \( H^2(X, W\mathcal{O}_X)\otimes_{W(k)}K \) is canonically identified with the slope-\( [0,1) \) part of the crystalline cohomology group \( H^2_{\mathrm{crys}}(X/W(k))\otimes_{W(k)}K \). For varieties \( X \) admitting a lifting to characteristic zero, they \cite[p. 121, Prop. 1.8]{ArtinMazurformal} also introduced an enlarged functor \( \Psi \), whose étale part corresponds to the divisible part of \( H^2_{\mathrm{fppf}}(X_{\bar{k}}, \mu_{p^\infty}) \), and whose connected part is \( \Phi^2 \). Since the connected-étale exact sequence for formal groups over a perfect field \( k \) splits, the Dieudonné module of \( \Psi \) is isomorphic, after tensoring with \( K \), to the slope-\( [0,1] \) part of \( H^2_{\mathrm{crys}}(X/W(k))\otimes_{W(k)}K \).

Artin and Mazur \cite[p. 91]{ArtinMazurformal} posed the question: \emph{Is there a direct relationship between the Dieudonné module of \( \Psi \) and the slope-\( [0,1] \) part of \( H^2_{\mathrm{crys}}(X/W(k))\otimes_{W(k)}K \)?}

For K3 surfaces \( X/k \) of finite height, Nygaard and Ogus \cite[Thm. 3.20]{NO} proved that \( \Psi \) is a \( p \)-divisible group, and its Dieudonné module is canonically isomorphic to the slope-\( [0,1] \) part of \( H^2_{\mathrm{crys}}(X/W(k)) \).

Bragg-Olsson \cite{bragg2021representabilitycohomologyfiniteflat} proved that $R^if_*\mu_{p^n}$ is representable by an affine, commutative group scheme of finite type over $k$. Motivated by the question of Artin--Mazur, we study the \( p \)-divisible part of the formal group \( R^i f_* \mu_{p^\infty} \) for a smooth projective variety over an arbitrary field of characteristic \( p \). Instead of computing its Dieudonné crystal directly, we show that there exists a canonical isomorphism in the isogeny category of formal groups between \( R^i f_* \mu_{p^\infty} \) and a \( p \)-divisible group associated to the slope-\( [0,1] \) part of the \( F \)-isocrystal \( R^i f_{\mathrm{crys}*} \mathcal{O}_{X/\mathbb{Z}_p} \). As a consequence, this yields a \emph{functorial isomorphism} between the Dieudonné crystal of the divisible part of \( R^i f_* \mu_{p^\infty} \) and the slope-\( [0,1] \) part of the \( F \)-isocrystal \( R^i f_{\mathrm{crys}*} \mathcal{O}_{X/\mathbb{Z}_p} \) in the category of \( F \)-isocrystals over \( \operatorname{Spec}(k)/\mathbb{Z}_p \).

\begin{theorem}[Theorem \ref{theoremmain}]
Let $S=\Spec(k)$, where $k$ is a field finitely generated over a perfect field of characteristic $p>0$, and let $f:X\to S$ be a smooth projective variety over $S$. Assume that $G$ is a p-divisible group over $X$. Then
    the formal group $R^if_{\fppf*}G$ is isogenous to a $p$-divisible group over $S$. Moreover, the divisible part $H^i$ of $R^if_{\fppf*}G$ is a $p$-divisible group, and there is natural isomorphism of $F$-isocrystals over $S$:
    $$\cM^{cr}(H^i)_{\bQ}\cong R^if_{\crys*}\cM^{cr}(G)_{\bQ,[0,1]}.$$
\end{theorem}
 Taking $G=\mu_{p^\infty}$ and $i=2$, this result answers a question (rationally) of \cite[p. 91]{ArtinMazurformal} on describing the enlarged functor, even when the base field is imperfect. For $i=1$ and $G=\mu_{p^\infty}$, the result follows from Oda's work \cite{oda1969first} when $k$ is perfect, and from Grothendieck–Messing theory \cite{messing1972crystals} in general.

 \begin{corollary}
     There is a natural isomorphism:
     $$(\varprojlim_n H^0_{\fppf}(S,R^if_{\fppf*}G[p^n]))_{\bQ}\cong H^0_{\crys}(S/W,R^if_{\crys*}\cM^{cr}(G))_{\bQ}^{\varphi=p}$$
 \end{corollary}

This gives a generalization and cleaner proof of \cite[Prop. 4.2]{li2024ptorsionsgeometricbrauergroups}, where the result was obtained via a pullback argument.
\subsection*{Notations}
Let $p$ be a prime number. Write $W=\bZ_p$, $\cI=(p)$, $W_n=\bZ_p/p^n\bZ_p.$

For a scheme $X$, denote by $\fppf(X)$, $\SYN(X)$, and $\syn(X)$ the big $\fppf$ site, the big syntomic site, and the small syntomic site of $X$, respectively (cf. \cite[Def.~1.3]{Bauer1992}). The associated topoi are denoted by $X_\fppf$, $X_\syn$, and $X_\SYN$.

Let $S=(S,I,\gamma)$ be a PD-scheme and $X$ be an $S$-scheme to which $\gamma$ extends. Denote the small crystalline(-Zariski) site by $\Crys(X/S)$ and the associated topoi by $(X/S)_{\crys}$. For $\tau \in \{\ZAR, \SYN \}$, let $\CRYS(X/S)_\tau$ be the big crystalline site with $\tau$-topology and denote the associated topoi by $(X/S)_{\CRYS,\tau}$. When working over $S = (\bF_p, \bZ_p, p)$, we will omit $S$ and simly write the site as $\CRYS(X)_{\tau}$ and the topos as $X_{\CRYS,\tau}$.

\subsection*{Acknowledgments}
We thank Marco D'Addezio and Zhenghang Du for their interest and communications. The first named author thank Yicheng Zhou for helpful discussions.

The paper is written when the first named author is in his Ph.D program in Sorbonne Universit\'e. This program has received funding from the European Union's Horizon 2020 research and innovation programme under the Marie Skłodowska-Curie grant agreement No 945332 and partially funded by the project “Group schemes, root systems, and related representations” founded by the European Union - NextGenerationEU through Romania’s National Recovery and Resilience Plan (PNRR) call no. PNRR-III-C9-2023- I8, Project CF159/31.07.2023.  The second named author is supported by the DFG through CRC1085 Higher Invariants (University of Regensburg).

\section{Preliminaries}
\subsection{Change of topology}

\begin{lemma}\label{fppfsyn}
 Let $G$ be a smooth or finite flat commutative group scheme over $X$ and 
 $$v: \fppf(X) \lra \SYN(X) $$ 
 be the morphism of sites. Then $R^qv_*G=0$ for any $q>0$.
\end{lemma}
\begin{proof}
The case $G$ is smooth follows from the the proof of \cite[Chap. III, 3.9]{Mil3}. It is in fact the syntomic sheafication on $\SYN(X)$ of the higher direct image from the $\fppf$ site to the big \'etale site. For $q>0$, the latter is 0. So $R^qv_*G$ is also $0$ for $q>0$.

When $G$ is finite flat, the question is local and we assume $X$ is affine. By \cite[Prop. 5.1]{Ma-Ro}, there exists an exact sequence of group scheme over $X$:
\begin{equation}\label{sequenceforfiniteflat}
    0\to G\to G_1\to G_2\to 0
\end{equation}
where $G_1$, $G_2$ are smooth of finite type over $X$. The morphism $G_1\to G_2$ is faithful flat so it is syntomic. The sequence is exact on $\SYN(S)$ because $G_1\to G_2$ is a syntomic cover. The conclusion follows from the smooth case.
\end{proof}
The map $f$ induces a commutative diagram of sites
\begin{displaymath}
\xymatrix{
\fppf(X)\ar[r]^v \ar[d]^f & \SYN(X) \ar[d]^f \\
\fppf(S)\ar[r]^v &  \SYN(S)
}
\end{displaymath}
By the Lemma above, $R^qv_*G[p^n]_X=0$ (also true on $S$) for any $q>0$. Thus, there is a Leray spectral sequence
$$E^{p,q}_2=R^pv_*R^qf_*G[p^n]_{X} \rightarrow R^{p+q}f_*(v_*G[p^n]_{X})$$

\begin{lemma}\label{fptosynvanishing}
Let $k$ be a field of characteristic $p>0$ and $G$ be a commutative group scheme of finite type over $S=\Spec(k)$. Let $v$ denote 
$$
 v: \fppf(S) \lra \SYN(S).
 $$ 
Then, $R^qv_*G=0$ for all $q>0$.
\end{lemma}
\begin{proof}
By \cite[Lem. 2.1.1]{tatedualitypositivedimension}, any commutative group scheme of finite type over a field is an extension of a smooth group scheme by a finite group scheme. The claim follows from Lemma \ref{fppfsyn}.
\end{proof}

\begin{corollary}\label{fppfsyncoho}
 For any $q\geq 0$, the edge map 
 $R^{q}f_*(v_*G[p^n]_{X})\lra v_*R^qf_*G[p^n]_{X}$ is an isomorphism.
\end{corollary}

\subsection{Isogeny and $F$-crystals}
 We first record a general lemma for abelian categories. It will be used to construct isogenies between $p$-divisible groups.

\begin{lemma}\label{findisogeny}
    Let $\sC$ be an abelian category, $\{A_n\}_{n\in \bN}$, $\{B_n\}_{n\in \bN}$ be two inducitve systems and $\{f_n:A_n\to B_n\}_{n\in \bN}$ are compatible maps between the systems such that $f_n$ has kernel and cokernel killed by $p^k$. Then there exist compatible maps $\{g_n: B_n\to A_n\}$ such that $f_ng_n=p^{2k}_{B_n}$ and $g_nf_n=p^{2k}_{A_n}$.
\end{lemma}

\begin{proof}
    Let $K_n=\ker(f_n)$. Since $p^kK_n=0$, by universal property, there exists a unique $\alpha_n:A_n/K_n\to A_n$ so that $A_n\stackrel{\pi_n}\to A_n/K_n\stackrel{\alpha_n}\to A_n$ is $p^k$. Since the cokernel is killed by $p^k$, the map $p^k: B_n\to B_n$ factors through $\im(f_n)$. As $\im(f_n)\cong \coim(f_n)$, there is a unique map $\beta_n: p^kB_n\to \im(f)\cong A_n/K_n$. Let $i_n: A_n/K_n\cong \im(f_n)\to B_n$. We define $g_n$ to be the composition
    $$g_n: B_n\stackrel{p^k}\to p^kB_n\stackrel{\beta_n}\to A_n/K_n\stackrel{\alpha_n}\to A_n.$$
    Use the epimorphism $\pi_n$ and the monomorphism $i_n$, we have $f_n\circ \alpha_n=p^ki_n$ and $\beta_n\circ p^k\circ f_n=p^k\pi_n$. Thus $g_n\circ f_n=\alpha_n\circ \pi_np^k=p^{2k}$, $f_n\circ g_n=i_np^k\circ \beta_np^k=p^{2k}$.
    
    Now we check compatibility
    $$\xymatrix{
    A_n \ar_{a_n}[ddd] \ar^{f_n}[rrr] & & & B_n\ar^{b_n}[ddd] \ar_{p^k}[ld]\\
     &A_n/K_n \ar_{\alpha_n}[lu] \ar_{a_n'}[d] & p^kB_n \ar_{\beta_n}[l] \ar^{b_n'}[d] & \\
     & A_{n+1}/K_{n+1} \ar^{\alpha_{n+1}}[ld] & p^kB_{n+1} \ar_{\beta_{n+1}}[l] & \\
     A_{n+1}\ar^{f_{n+1}}[rrr] & & & B_{n+1} \ar^{p^k}[lu]
    }$$
    The left trapezoidal commutes because of universal property of $\alpha_n$: $a_n\circ\alpha_n$ is the unique map such that $a_n\circ \alpha_n\circ \pi_n=a_np^k$. But $\alpha_{n+1}\circ a_n'\circ\pi_n=\alpha_{n+1}\circ\pi_{n+1}\circ a_n=p^ka_n$. So $a_n\circ\alpha_n=\alpha_{n+1}\circ a_n'$. The middle square commutes because $a_n'\circ\beta_n$ is the unique map such that $i_n \circ a_n'\circ\beta_n=b_n\circ j_n$, but $i_n\circ \beta_{n+1}\circ b_n'=b_n'\circ j_{n+1}=j_n\circ b_n$. So $a_n'\circ\beta_n=\beta_{n+1}\circ b_n'$. The right trapezoidal commutes naturally. A diagram chasing tells $a_n \circ g_n=g_{n+1}\circ b_n$
\end{proof}

Let $S$ be a scheme and $\sF\in S_{\fppf}$ be a sheaf in abelian groups over the fppf site of $S$. We fix a prime $p$ and call $\sF$ divisible if the map $p: \sF\to \sF$ is surjective.

\begin{lemma}\label{lemmaconvenient}
    If $f:\sF\to \sG$ is an injective map in $S_{\fppf}$ such that $\sG$ is divisible and $\coker(f)=:\sH$ is annihilated by $p^N$, then $f$ is an isomorphism.
\end{lemma}

\begin{proof}
    The induced map $\coker(p^n_\sG)\to \coker(p^n_\sH)$ is surjective. When $n>N$, $\coker(p^n_\sG)=0$ but $\coker(p^n_{\sH})=\sH$.
\end{proof}

\begin{lemma}\label{lemmadivisible}
    Assume $\sF$ is isogeny to a $p$-divisible group $\sG$, regarded as an fppf sheaf over $S$. Then $p^n\sF$ is divisible for any $n$ large enough.
\end{lemma}

\begin{proof}
    Suppose $f:\sF\rightleftarrows\sG: g$ such that $fg=p^N$, $gf=p^N$. We have $p^N\sF\subseteq \im(g)$ whose cokernel is annihilated by $p^N$. Note that $\im(g)$ is divisible. Lemma \ref{lemmaconvenient} tells $p^n\sF\cong \im(g)$ is divisible for all $n\geq N$.
\end{proof}

\begin{lemma}\label{lemmapdivisible}
      Assume $S=\Spec k$ is the spectrum of a field and $\sF$ is isogeny to a $p$-divisible group $\sG$. Suppose $\sF\cong \colim_n \sF_n$ such that $\sF_n$ is a $p^n$-torsion sheaf represented by a group scheme of finite type. Then $p^N\sF$ is a $p$-divisible group for any $N$ large enough.
\end{lemma}

\begin{proof}

   Suppose $f:\sF\rightleftarrows\sG: g$ such that $fg=p^N$, $gf=p^N$. By lemma \ref{lemmadivisible}, it is enough to show $\im(g)$ is a $p$-divisible group and it is left to show $\im(g)[p]$ is represented by a finite flat group scheme. We have an exact sequence
   $$0\to \ker(g)[p]\to \sG[p]\to \im(g)[p]\to \coker(p_{\ker(g)})\to 0.$$
   The category of commutative finite flat group schemes over a field is an abelian category, so it suffices to show $\ker(g)$ is a finite flat group scheme.
   Note that $\ker(g)\subseteq \sG[p^N]$. Regard $\sG[p^N]$ as an object in $\fppf(S)$, it satisfies the condition (4) of \cite[\href{https://stacks.math.columbia.edu/tag/0738}{Tag 0738}]{stacks-project}, so $\sF(\sG[p^N])\cong \colim_i (\sF_i(\sG[p^N]))$. It means $g:\sG[p^N]\to \sF$ factors through some $g_i: \sG[p^N]\to \sF_i$. Since $\sG[p^N]$ is the spectrum of a finite dimensional $k$-vector space, $\ker(g_j: \sG[p^N]\to \sF_j)$ stabilizes where $j\geq i$. We get $\ker(g)=\ker(g_j)$ for $j$ large enough and $\ker(g_j)$ is a finite flat group scheme.
    
\end{proof}

\begin{definition}\label{defpdivisible}
    Under the assumption of lemma \ref{lemmapdivisible}. We define $\Div(\sF)$, the $p$-divisible part of $\sF$, to be $p^n\sF\subseteq \sF$ where $n$ is a large enough integer. The formation is functorial and does not depend on the choice of $n$.
\end{definition}

Let $S$ be a Noetherian scheme of characteristic $p$. An $F$-crystal over $S$ is an $\cO_{S/\bZ_p}$-coherent crystal $\cE$ with a morphism $F:\cE^{(\sigma)}\to \cE$ such that the kernel and cokernel of $F$ are annihilated by a power of $p$, where $\cO_{S/\bZ_p}$ is the structure sheaf of the big crystalline site of $S$ over $\bZ_p$. By \cite[Lemma 6.1]{li2024ptorsionsgeometricbrauergroups}, this is an abelian category when $S$ is a finitely generated field in positive characteristic. The category of $F$-isocrystals is the isogeny category of $F$-crystals.

Now we assume $S=\Spec k$ is a field in characteristic $p$. Let $\cE$ be an $F$-crystal over $S$. By \cite[Claim 2.8]{de2000purity}, there is a locally free $F$-crystal $\cE'$ isogeny to $\cE$ such that $\cE'$ admits a slope filtration
    $$0\subset\cE'_{1}\subset \cE'_2\subset \cdots \cE'.$$
    Here each $\cE'_i/\cE'_{i-1}$ is divisible by $\lambda_i\in \bQ_{\geq 0}$ and isoclinic of slope $\lambda_i$ where $\lambda_i\leq\lambda_j$ when $i\leq j$. Let $\cE'_{[0,1]}\subseteq \cE'$ be the locally free $F$-crystal in the filtration with slope between $[0,1]$ and all slopes of the locally free $F$-crystal $\cE'/\cE'_{[0,1]}=:\cE'_{>1}$ are greater than $1$. For each $F$-crystal $\cE$, we fix such an $\cE'$ and an isogeny $i_{\cE}:\cE'\to \cE$. It induces a map $\cE'_{[0,1]}\to \cE'\to \cE$. 
    \begin{definition}\label{defslopepart}
        Passing the above map to the isogeny category, we define $\cE'_{[0,1],\bQ}\to \cE_{\bQ}$ to be the slope $[0,1]$-part of the $F$-isocrystal $\cE_\bQ$. This formation is functorial in the category of $F$-isocrytals (but not functorial in the category of $F$-crystals).
    \end{definition}

\section{The higher direct image}
\subsection{Syntomic complex}
We recall the syntomic complex constructed in \cite{trihan2018comparisontheoremsemiabelianschemes}.

Let $\cM^{cr}(G)$ (resp. $\cM^{cr}(G[p^n])$ be the (covariant) Dieudonne crystal attached to the $p$-divisible group $G$ (resp. finite flat group scheme $G[p^n]$), regarded as a flat coherent $F$-crystal over $\CRYS(X)_{\SYN}$, the big crystalline syntomic site of $X$.
Consider the following commutative diagram of topos
    $$\xymatrix{
    X_{\CRYS,\SYN}\ar^u[r] \ar_{f_{\crys}}[d] & X_{\SYN}\ar_{f}[d] \\
    S_{\CRYS,\SYN}\ar_u[r] & S_{\SYN}.
    }$$
Let $\cM(G)$ (resp. $\cM(G[p^n])$) be $u_*\cM^{cr}(G)$ (resp. $u_*\cM^{cr}(G[p^n])$. We will use the same notation to denote their restriction to the small syntomic site when the background indicates the site we are considering.

\begin{lemma}\label{lemmasomecollection}
    \begin{enumerate}
        \item The sheaf $R^if_{\crys*}\cM^{cr}(G)$ is isogenous to an locally free $F$-crystal.
        \item The spectral sequence $$R^iu_*R^jf_{\crys*}\cM^{cr}(G[p^n])\Rightarrow R^{i+j}f_*u_*\cM^{cr}(G[p^n])$$
        induces an edge map $R^if_*u_*\cM^{cr}(G[p^n])\to u_*R^if_{\crys*}\cM^{cr}(G[p^n])$. The edge map is an al. isomorphism uniformly in $n$.
    \end{enumerate}  
\end{lemma}

\begin{proof}
    By \cite[Prop. 1.17]{Bauer1992}, if $\cF$ is a quasi-coherent crystal then $Ru_*\cF=u_*\cF$. This induces the spectral sequence in (2) and the claim on edge map follows from (1).

    For the first claim, by \cite[Lemma 4.12]{li2024ptorsionsgeometricbrauergroups}, it suffices to prove the property on the big crystalline-Zariski site. The proof is similar to \cite[Prop 4.14]{li2024ptorsionsgeometricbrauergroups}: Let $C=C(k)$ be a Cohen ring of $k$, $D=\Spec C$ and write $\cM=\cM(G)$. By \cite[Thm. 7.24]{BerthelotOgus1978}, $R\Gamma_{\crys}(X/D,\cM)$ is a perfect $C$-complex with finite Tor-amplitude. So we can assume $p^NH^n_{\crys}(X/D,\cM)$ is a finite free $C$-module for all $n$. Use the same argument as in \cite[Appendix A]{Mor}, $p^NH^i_{\crys}(X/D,\cM)$ induces a crystal $\cE$ whose value on $(U,T)$ with $U,T$ affine is $g_T^*p^NH^i_{\crys}(X/D,\cM)$ where $g_T:T\to D$ is a lifting of $U\to S$ (by quasi-smoothness of $C$ over $W$). The Frobenius structure on $H^i_{\crys}(X/D,\cE)$ induces an $F$-crystal structure on $\cE$. The map
    $$\cE_T=g_T^*p^NH^i_{\crys}(X/D,\cM)\to g_T^*H^i_{\crys}(X/D,\cM)\to H^i(Lg_T^*R\Gamma_{\crys}(X/D,\cM))=(R^if_{\crys*}\cM)_T$$
    is $F$-equivariant and has a kernel and cokernel annihilated by $p^{2N}$. So $\cE$ is an $F$-crystal isogeneous to $R^if_{\crys*}\cM^{cr}(G)$.
\end{proof}

By definition and \cite[Lem 9.2]{trihan2018comparisontheoremsemiabelianschemes}
$$\cM(G[p^n])=u_*\cE xt^1_{X_{\CRYS,\SYN}}(u^{-1}G[p^n]^*,\cO_n)\cong \cE xt^1_{X_{\syn}}(G[p^n]^*,\cO^{\crys}_n).$$
There is a mod-$p$ Hodge filtration $\Fil^1\cM(G)$ on $\cM(G)$ defined in \cite[Sec. 5.4]{trihan2018comparisontheoremsemiabelianschemes}. It has the following interpretation: Let $\cJ_n:=\ker(\cO_n^{\crys}\to \bG_a)$ as a sheaf on $\syn(X)$ and $\cI_n:=\cJ_{n+1}/p^{n}$, then $\Fil^1\cM(G[p^n])\cong \cE xt^1_{X_{\syn}}(G[p^n]^*,\cI_n)$.
The sheaf $\cI_n$ is equipped with a 'divided Frobenius' $\varphi': \cI_n\to \cO_n^{\crys}$ such that the following diagram commutes
\begin{equation}\label{diagramfrobcomm}\begin{CD}
    \cI_n @>\varphi' >> \cO^{\crys}_n\\
    @Vi VV @VVp V\\
    \cO^{\crys}_n@>>\varphi > \cO^{\crys}_n.
\end{CD}\end{equation}
Moreover, we have exact sequences in $X_{\syn}$ (cf. \cite[Lem. 9.9]{trihan2018comparisontheoremsemiabelianschemes}):
$$\begin{CD}
    0 @>>> \mu_{p^n} @>>> \cI_n@>\varphi'-1>> \cO^{\crys}_n @>>> 0
\end{CD}$$
The following theorem is deduced by applying $\cE xt^1_{X_{\syn}}(G[p^n]^*,-)$ to the above exact sequence.

\begin{theorem}[{\cite[Thm. 9.13]{trihan2018comparisontheoremsemiabelianschemes}}]\label{syntomiccomplex}
There exist compatible exact sequences in $X_{\syn}$
\begin{equation}\label{formulasyntomiccomplex}
    0\to G[p^n]\to \Fil^1\cM(G[p^n])\stackrel{\varphi'-1}\longrightarrow \cM(G[p^n])\to 0
\end{equation}
\end{theorem}

\begin{remark}\label{remarkshiftfrob}
    Apply $\cE xt^1_{X_{\syn}}(G^*,-)$ to the diagram \eqref{diagramfrobcomm}, we get a commutive diagram
    $$\begin{CD}
        \Fil^1\cM(G)@>\varphi'_{\cM}-1 >> \cM(G)\\
        @Vi VV @VV p V\\
        \cM(G)@>>\varphi_{\cM}-p> \cM(G).
    \end{CD}$$
    The map $i$ is injective with cokernel annihilated by $p$. Then it is not hard to see that kernel and cokernel of the maps $\varphi'_{\cM}-1$ and $\varphi_{\cM}-p$ are isomorphic up to a group annihilated by $p^2$.
\end{remark}

 \subsection{Splitting of the long exact sequence} Let $S=\Spec k$ where $k$ is a finitely generated field of characteristic $p$ and $f:X\to S$ is a smooth projective variety over $S$. Assume $G$ is a p-divisible group over $X$. We have a short exact sequence \eqref{formulasyntomiccomplex} and push it forward along $f$, we get an exact sequence in $S_{\syn}$
    \begin{equation}\label{formulapushsyntomic}
        \begin{aligned}
        R^{i-1} f_{\syn*}\Fil^1\cM(G[p^n]) &\stackrel{{\varphi'_{i-1,n}-1}}\lra  R^{i-1}f_{\syn*}\cM(G[p^n])\to R^if_{\syn*}G[p^n]\to  \\
        &R^{i}f_{\syn*}\Fil^1\cM(G[p^n])\stackrel{\varphi'_{i,n}-1}\lra R^{i}f_{\syn*}\cM(G[p^n]).
        \end{aligned}\end{equation}

We show this long exact sequence almost splits.

\begin{proposition}\label{propmain1}
    The sheaf $\coker(\varphi'_{i-1,n}-1)$ is al. zero uniformly in $n$.
\end{proposition}

\begin{proof}
     Lemma \ref{lemmachangebase} and \ref{lemmatotorsion}  reduces the question to proposition \ref{propcokernelfinite}.
\end{proof}

\begin{lemma}\label{lemmatotorsion}
     Assume $\varphi-p: R^if_*\cM(G)\to R^if_*\cM(G)$ has a cokernel of finite exponent, then $\coker(\varphi_{i,n}'-1)$ is al. zero uniformly in $n$.
\end{lemma}

\begin{proof}
    By remark \ref{remarkshiftfrob}, it suffices to show $\coker(\varphi_{i,n}-p:\cM(G[p^n])\to \cM(G[p^n]))$ is al. zero uniformly in $n$. We have an exact sequence of sheaves over $\syn(X)$: $0\to \cM(G)\stackrel{p^n}\to \cM(G)\to \cM(G[p^n])\to 0$. Pushforward along $f$, we get an exact sequence
    $$\begin{aligned}
        0\to (R^if_*\cM(G))[p^n]\to R^if_*\cM(G)\stackrel{p^n}\to R^if_*\cM(G)\to\\
        R^if_*\cM(G[p^n])\to (R^{i+1}f_*\cM(G))[p^n]\to 0
    \end{aligned}$$
    Both $R^if_{*}\cM(G)$ and $R^{i+1}f_{*}\cM(G)$ has bounded $p$-torsions, so the map $R^if_*\cM(G)/p^n\to R^if_*\cM(G[p^n])$ is an al. isomorphism uniformly in $n$. The claim follows from snake lemma.
\end{proof}

\begin{lemma}\label{descenttok}
Let \( S = \Spec(k) \), and let \( \cF \) be a presheaf on \( \SYN(S) \) satisfying the following property: for any quasi-compact and quasi-separated \( k \)-scheme \( T \),
\[
\cF(T_{\bar{k}}) = \colim_{[l:k]<\infty,\, l \subset \bar{k}} \cF(T_l).
\]
Define the presheaf \( \cF^+ \) by setting \( \cF^+(U) := \check{H}^0(U, \cF) \) for any \( k \)-scheme \( U \) (cf.  \cite[\href{https://stacks.math.columbia.edu/tag/00W1}{Tag 00W1}]{stacks-project}). Then \( \cF^+ \) also satisfies the same property.
\end{lemma}

\begin{proof}
Let \( T \) be a quasi-compact and quasi-separated \( k \)-scheme, and consider a syntomic cover \( U_{\bar{k}} \to T_{\bar{k}} \). Since \( T \) is quasi-compact, we may refine this cover so that \( U_{\bar{k}} \) is affine. Moreover, we may assume that the cover \( U_{\bar{k}} \to T_{\bar{k}} \) arises as the base change of a cover \( U_l \to T_l \), where \( l \subset \bar{k} \) is a finite extension of \( k \) (cf. \cite[Rmk.10.68]{goertz-wedhorn2020} and \cite[\href{https://stacks.math.columbia.edu/tag/00SM}{Tag 00SM}]{stacks-project} ). It then follows that
\[
\check{H}^0(U_{\bar{k}} \to T_{\bar{k}}, \cF) = \colim_{l'/l} \check{H}^0(U_{l'} \to T_{l'}, \cF),
\]
where the colimit is taken over finite extensions \( l'/l \subset \bar{k} \). Since \( \check{H}^0(T_{\bar{k}}, \cF) \) is the colimit over all such Čech 0-th cohomologies \( \check{H}^0(U_{\bar{k}} \to T_{\bar{k}}, \cF) \), ranging over a cofinal system of covers of \( T_{\bar{k}} \), the result follows.
\end{proof}

We will consider the following diagrams:
$$\begin{CD}		\overline{X}_{\CRYS,\SYN}@>u_{\overline{X}}>>\overline{X}_{\SYN}@>\overline{f}_{\SYN}>> \overline{S}_{\SYN}@>\overline{v}_S>> \overline{S}_{\syn}\\
		@V\pi_0VV @V\pi_1VV @VV\pi_2V @VV\pi_3V\\
		X_{\CRYS,\SYN}@>>u_X> X_{\SYN}@>>f_{\SYN}> S_{\SYN} @>>v_S> S_{\syn}
	\end{CD}$$

    $$\begin{CD}
		X_{\CRYS,\SYN} @>u_X>> X_{\SYN}@>v_X>> X_{\syn}\\
		@Vf_{CR,\SYN}VV @VVf_{\SYN}V @VVf_{\syn}V \\
		S_{\CRYS,\SYN} @>>u_S> S_{\SYN}@>>v_S> S_{\syn} \\
	\end{CD}$$

 \begin{lemma}\label{lemmachangebase}
Let $\bar{S}=\Spec(\bar{k})$ and $\overline{\varphi}^i-p: R^i\overline{f}_*\cM(\overline{G})\to R^i\overline{f}_*\cM(\overline{G})$ be the morphism between sheaves over small syntomic site of $\overline{S}$ constructed from $X_{\overline{S}}\to \overline{S}$.
If the sheaf $\coker(\overline{\varphi}^i-p)$ is of finite exponent, then $\coker(\varphi^i-p)$ is also of finite exponent.
\end{lemma}

\begin{proof}
	Both maps $\overline{\varphi}^i-p$ and $\varphi^i-p$ are push forward of corresponding maps on the big syntomic site to small syntomic site. Write $\overline{\cE}=R^i\overline{f}_{\SYN*}\cM(\overline{G})$ and $\cE=R^if_{\SYN*}\cM(G)$ as sheaves over big syntomic sites. There is a natural map $\cE\to \pi_{2*}\pi_2^{-1}\cE$.  Note that we have $\pi_1^{-1}\cM(G)=\cM(G)|_{\SYN(\overline{X})}=\cM(\overline{G})$. It is enough to see  $\pi_0^{-1}\cM^{cr}(G)=\cM^{cr}(\overline{G})$. For any $(U,T)\in \CRYS(\overline{X})$, one has 
		  $$\pi^{-1}_1\cM(G)(U,T)=\Ext^1_{X}(u^{-1}G,\cO_T)=\Ext^1_{\overline{X}}(u^{-1}\overline{G},\cO_T)=\cM(\overline{G})(U,T).$$
So we have isomorphisms $\pi_2^{-1}\cE= \pi_2^{-1}R^if_*\cM(G)\cong R^i\overline{f}_*\pi^{-1}_1\cM(G)\cong R^i\overline{f}_*\cM(\overline{G})=\overline{\cE}$.

		 There is a commutative diagram where the first rows is exact
	 
	 	$$\xymatrix{
	\cE \ar^{\varphi^i-p}[r]\ar[d]& \cE \ar[d]\ar[r]&\coker(\varphi^i-p)\ar[d]\ar[r] & 0\\
	\pi_{2*}\overline{\cE} \ar^{\overline{\varphi}^i-p}[r]& \pi_{2*}\overline{\cE}\ar[r]& \pi_{2*}\coker(\overline{\varphi}^i-p)
	}$$

	We may replace $\cE$ by $\cF:=u_{S*}\cE'_{cr}$ and $\overline{\cE}$ by $\overline{\cF}:=u_{\overline{S}*}\pi^{-1}_{cr}\cE'_{cr}$ where $\cE_{cr}'$ is a $F$-crystal isogeny to $R^if_{CR,\SYN*}\cM^{cr}(G)$ (cf. lemma \ref{lemmasomecollection}) and $\pi_{cr}:\overline{S}_{CR,\SYN}\to S_{CR,\SYN}$. Write $\cC:=\coker(\varphi^i-p)$ and $\overline{\cC}:=\coker(\overline{\varphi}^i-p)\cong \pi_2^{-1}\cC$. The condition says that $v_{\overline{S}*}\overline{\cC}$ has a finite exponent, so is $\pi_{3*}v_{\overline{S}*}\overline{\cC}=v_{S*}\pi_{2*}\overline{\cC}$.  Consider the following diagram

	$$\xymatrix{
	v_{S*}\cF \ar^{\varphi^i-p}[r]\ar[d]& v_{S*}\cF \ar[d]\ar[r]&v_{S*}\cC \ar^\theta[d]\ar[r] & 0\\
	\pi_{3*}v_{\overline{S}*}\overline{\cF} \ar^{\overline{\varphi}^i-p}[r]& \pi_{3*}v_{\overline{S}*}\overline{\cF}\ar[r]& \pi_{3*}v_{\overline{S}*}\overline{\cC}
	}$$

	 It suffices to show the map $\theta:v_{S*}\cC\to \pi_{3*}v_{\overline{S}*}\pi_2^{-1}\cC$ is injective. Write $\overline{k}=\colim_{l\in \bN} k_l$ as colimit of finite extension of $k$. Note that $\cF$ has the property that $\cF(T_{\bar{k}})=\colim_l \cF(T_{k_l})$ for any $T\in \SYN(S)$ affine (cf. \cite[2.4.3]{kato1988expose}). By lemma \ref{descenttok}, $\cC$ also has this property. For any $T\in \syn(S)$ affine, the map $\theta(T)$ is given by 
     $$\cC(T)\to \pi_{3*}v_{\overline{S}*}\pi_2^{-1}\cC(T)=\cC(T_{\overline{k}})=\colim_l \cC(T_{k_l}).$$
      For each $l\leq l'$, the map $T_{k_{l'}}\to T_{k_l}$ is a syntomic covering, so all transition maps in the colimit system are injective. It follows the map $\theta(T)$ is injective.
     \end{proof}

We recall the following lemma (See also \cite[Rmk. 4.1.9]{scholze2013moduli} and \cite[Sec. 2.5]{drinfeld2018theorem})

\begin{lemma}\label{lemmaprepare1}
    Assume $R=P/J$ is a regular semiperfect ring in characteristic $p$, where $P=k[X_1^{p^{-\infty}},...,X_d^{p^{-\infty}}]$ is a perfect $k$-algebra and $J$ is an ideal of $P$ generated by a regular sequence $f_1,...,f_r$. Then $A_{\crys}(R)=(W(P)[ \frac{[{f_i}]^k}{k!}]_{k\geq 1})^{\wedge p}$. It contains elements
    $$b=\sum_{\underline{k}} a_{\underline{k}}\gamma_{\underline{k}}([{f}])$$
	where $\underline{k}=(k_1,...,k_r)\in \bN^r$, $a_{\underline{k}}\in W(P)$ and $\gamma_{\underline{k}}([f]):=\prod_{i=1}^r \gamma_{k_i}([{f_i}])$ such that $a_{\underline{k}}$ convergents to $0$ in $W(P)$ when $|\underline{k}|=\sum_{i=1}^r k_i$ tends to infinity.
\end{lemma}

\begin{proof}
    The first statement follows from a direct computation (For example cf. \cite[Sec. 2.5]{drinfeld2018theorem}) or we can use prismatic cohomology: Let $J=([f_1],...,[f_r])$, then 
    $$D_J(W(P))^\wedge=W(P)\{\sigma([f_i])/p\}_\delta^{\wedge}=\sigma^*\Delta_{R/\bZ_p}\cong A_{\crys}(R).$$
    For the second claim, it is clear that such $b$ is an element of $D_J(W(P))^\wedge$. To see every element of $D_J(W(P))^\wedge$ can be (non uniquely) written as such a form, Let $T:=W(P)\langle x_1,...,x_r\rangle$ be the divided power polynomial over $W(P)$ with $r$ variables. There is a surjective map $T\to D_J(W(P))$ mapping $x_i$ to $f_i$. It suffices to show the map is surjective after $p$-adic completion since every element in $T^{\wedge}$ contains elements of the form of $b$ replacing $\gamma_{\underline{k}}([f])$ by $\gamma_{\underline{k}}([x])$. This follows from Mittage-Leffler condition of the system $\{\ker(T/p^n\to D_J(W(P))/p^n)\}_n$ (this is a surjective system).
\end{proof}

\begin{lemma}\label{lemmanegativeslope}
	 Let $m,n$ be two positive integers. Then there is a constant $N=N(m,n)$ such that for all regular semiperfect ring $R=P/J$ over $k$ as in lemma \ref{lemmaprepare1}, then
	$$A_{\crys}(R)\stackrel{\sigma^m-p^n}\longrightarrow A_{\crys}(R)$$
	has a cokernel killed by $p^N$.
\end{lemma}

\begin{proof}
    Recall that $\sigma^m-p^n$ is surjective over $W(P)$ because $\sigma$ is invertible over $W(P)$ and for any $b\in W(P)$, the element 
	$$a:=\sigma^m(1-\sigma^{-m}p^n)^{-1} b=\sigma^m(b+\sigma^{-m}p^nb+(\sigma^{-m}p^n)^2b+\cdots)$$
	is a well-defined element in $W(P)$ with $|a|_p\leq |b|_p$. 
	 Let $M$ be the sub $W(P)$-module of $A_{\crys}(R)$ containing those elements $b=\sum_{|\underline{k}| > n/m} a_{\underline{k}}\gamma_{\underline{k}}([{f}])$. Since 
	$$\sigma\gamma_k([{f_i}])= \sigma([{f_i}]^k/k!)=[{f_i}]^{kp}/k!=\frac{(kp)!}{k!}\frac{[{f_i}]^{kp}}{(kp)!}=p^k g\gamma_{kp}([{f_i}])$$
	for some $g\in W(P)$, we find $M$ is stable under action of $p^{-n}\sigma^m$ and for any $x\in M$, $(p^{-n}\sigma^m)^ix$ tends to $0$ when $i$ tends to $0$. So  $-p^n(1+\sum_{i=1}^\infty (p^{-n}\sigma^m)^{i})\gamma_{\underline{k}}([{f}])$ is an well-defined element in $A_{\crys}(S)$ when $|\underline{k}|>n/m$ which means $M\subseteq \im(\sigma^m-p^n)$. We now consider $1\leq |\underline{k}|\leq n/m$. Let $N=n^2$, then $(p^{-n}\sigma^m)^n(p^N\gamma_{\underline{k}}([{f_i}])=\sigma^{mn}(\gamma_{\underline{k}}([{f_i}])\in M$ and $(p^{-n}\sigma^m)^r(p^N\gamma_{\underline{k}}([{f_i}])\in A_{\crys}(R)$ for all $1\leq r< n$ and $1\leq |\underline{k}|\leq n/m$. It implies $-p^n(1+\sum_{i=1}^\infty (p^{-n}\sigma^m)^{i})\gamma_{\underline{k}}([{f}])$ is a well-defined element in $A_{\crys}(R)$. So $p^NA_{\crys}(S)\subseteq \im(\sigma^m-p^n)$.
\end{proof}

\begin{proposition}\label{propcokernelfinite}
    Assume $S=\Spec k=\Spec \bar{k}$ and $ \cE$ is a non-degenerated coherent $F$-crystal on $\CRYS(S/W)_{\SYN}$. Let $u: (S/W)_{\CRYS,\SYN}\to S_{\syn}$ be the map of topos, then the map of abelian sheaves over $\syn(S)$
    $$\varphi_\cE-p:u_*\cE\to u_*\cE$$
    has a cokernel of finite exponent.
\end{proposition}
\begin{proof}
     Because $k$ is perfect, the crystalline site contains a final object $W(k)$ and the $F$-crystal $\cE$ is determined by the value on $W(k)$. For two integers $m\geq 0$ and $n>0$, let $D_{m,n}$ be the $F$-crystal on $\Crys(S/W)_{\syn}$ whose value on $W(k)$ is a $W(k)$-module free of rank $n$ with basis $e_0,...,e_{n-1}$ such that $\varphi(e_i)=e_{i+1}$ for $i<n-1$ and $\varphi(e_{n-1})=p^me_0$. Dieudonne-Manin classification tells that there exists a finite direct sum of $D_{m_i,n_i}$'s, for some $m_i,n_i$ coprime to each other, such that $\cE$ and $\oplus_i D_{m_i,n_i}$ are isomorphic in the isogeny category of $F$-crystals. By a diagram chasing, it suffices to show the claim for each $D_{m,n}$. 

For a regular semiperfect ring $R=P/J$ over $k$ where $P=k[X_1^{p^{-\infty}},...,X_d^{p^{-\infty}}]$ is a perfect $k$-algebra and $J$ is an ideal of $P$ generated by a regular sequence $f_1,...,f_r$. We have $(u_*D_{m,n})(R)\cong A_{\crys}(R)^n$ with basis $e_0,...,e_{n-1}$ such that $\varphi(e_i)=e_{i+1}$ for $i<n-1$ and $\varphi(e_{n-1})=p^me_0$.  Let $b_0,...,b_{m-1}\in A_{\crys}(R)$, we want to solve the equation 
$$p^m\sigma(x_{n-1})-px_0=b_0;\ \sigma(x_{i-1})-px_{i}=b_i\ (1\leq i\leq n-1).$$
Thus $x_i=p^{-i}\sigma^{i}(x_0)-(\sum_{k=1}^{i} p^{k-i-1}\sigma^{i-k}(b_k))$ for $1\leq i\leq n-1$ and $x_0$ satisfies $p^{m-n}\sigma^m(x_0)-x_0=p^{-1}b_0+\Sigma_{k=1}^{n-1} p^{k+m-n-1}\sigma^{n-k}(b_k)=:c$. Note that in this case, $A_{\crys}(R)$ is $p$-torsion free (cf. \cite[Lem. 6.1.9]{matthiassyntomic}).

Assume $m\neq n$, we show that there is a constant $N=N(m,n)$ such that for all regular semiperfect ring $R$ over $k$, the map $\varphi_{\cM}-p: D_{m,n}(R)\to D_{m,n}(R)$ has a cokernel killed by $p^N$. If $m-n>0$, let $N=2(m+n)+1$ and assume $b_i\in p^NA_{\crys}(R)$. Then $c\in p^{m+n}A_{\crys}(R)$ and $x_0=-(1+\sum_{k=0}^\infty (\sigma^mp^{m-n})^k)(c)\in p^{m+n}A_{\crys}(R)$ thus each $x_i$ is well-defined and we find the solution. If $m-n<0$, we need to solve $\sigma^m(x_0)-p^{n-m}x_0=p^{n-m}c$. By Lemma \ref{lemmanegativeslope}, there exists $N$ such that we can find a solution $x_0\in p^{m+n}A_{\crys}(R)$ whenever $b_i\in p^NA_{\crys}(R)$, thus each $x_i$ is well-defined and we win.

The case left is $m=n=1$. For any $R$ as above and $b\in pA_{\crys}(R)$ we need to solve $p\sigma(x)-px=b$. We can write $b=p(a_0+\sum_{|i|\geq 1} a_i\gamma_i(f))$. By computation in Lemma below, $\sigma\gamma_i(f)=p^ig\gamma_{pi}(f)$ for some $g\in W(R)$. We find $(1-\sigma)^{-1}(\sum_{|i|\geq 1}a_i\gamma_i(f))$ is well-defined. Thus we can find $R'$ an \'etale (thus syntomic) $R$-algebra ($R'$ is automatically semiperfect) such that there exists $x\in A_{\crys}(R')$ and $\sigma(x)-x=b$. It means in this case $\coker(D_{1,1}\stackrel{\varphi_M-p}\longrightarrow D_{1,1})$ is killed by $p$ and we can conclude.
\end{proof}

\subsection{Proof of the main theorem}

\begin{theorem}\label{theoremmain}
    The formal group $R^if_{\fppf*}G$ is isogenous to a $p$-divisible group $H^i$ such that we have an isomorphism of $F$-isocrystals over $S$:
    $$\cM^{cr}(H^i)_{\bQ}\cong R^if_{\crys*}\cM^{cr}(G)_{\bQ,[0,1]}.$$
\end{theorem}

\begin{proof}
    
    We consider the big crystalline-syntomic site of $S$. Note that the category of $F$-crystals on the (big) crystalline-syntomic site and the (big) crystalline-Zariski site coincides and admit the same cohomologies (\cite[1.18, Prop. 1.19]{berthelot2006theorie}). By Lemma \ref{lemmasomecollection} (1), $R^if_{\crys*}\cM^{cr}(G)$ is isogenous to an (nondegenerated) $F$-crystal $\cM'$. By definition \ref{defslopepart}, we have chosen a map $\cM_{[0,1]}\to \cM\to \cM'$.
        Moreover, by \cite[Lemma 5.8(2)]{Pal} (although the result is stated to smooth varieties, the proof still works for finitely generated field, noticing that de Jong's essentially surjectivity holds for finitely generated fields \cite[Theorem 1]{de1995crystalline}), the $F$-crystal $\cM_{[0,1]}$ is isogenous to an $F$-crystal $\cM(H^{i,*})$ which comes from a Dieudonne crystal associated to a $p$-divisible group $H^i$. After choosing these isogenies, we obtain a commutative diagram 
    \begin{equation}\label{formulaslope}
    \xymatrix{
    \ker(\varphi_{H^i}-p)\ar[r] \ar[d] & \cM(H^{i}) \ar^{\varphi_{H^i}-p}[r]\ar_j[d]& \cM(H^{i})\ar^j[d]\\
    \ker(\varphi_{\cM}-p)\ar[r] & R^if_{\crys*}\cM^{cr}(G) \ar_{\varphi_{\cM}-p}[r]& R^if_{\crys*}\cM^{cr}(G).
    }\end{equation}
    The kernel of $j$ is isogeny to $0$ and cokernel of $j$ is isogeny to $\cM_{>1}$. Suppose $1+\delta$ ($\delta\in \bQ_{>0}$) is the smallest slope of $\cM_{>1}$. By isogeny theorem (\cite[Thm. 2.6.1]{katzslope}, see \cite[proof of Claim 2.8]{de1999dieudonne}), $\cM_{>1}$ is isogenous to an $F$-crystal divisible by $1+\delta$ and we assume itself is. So there are positive integers $a,b$ and $F':\cM_{>1}^{(a)}\to \cM_{>1}$ such that $F_{\cM_{>1}}^a: \cM^{(a)}_{>1}\to \cM_{>1}$ satisfies $F^a_{\cM_{>1}}=p^{a+b}F'$. It implies $\varphi_{\cM_{>1}}^a-p^a:\cM_{>1}\to \cM_{>1}$ is invertible with inverse $p^a(1+p^b\varphi'+(p^b\varphi')^2+\cdots)$ and we can deduce that $\varphi_{\cM_{>1}}-p$ is an isomorphism. 
    Thus, the diagram \eqref{formulaslope} induces an isogeny between the kernel of the two rows.
    
    By Theorem \ref{syntomiccomplex} and remark \ref{remarkshiftfrob}, $\coker(\varphi_{H^i}-p)$ is killed by $p^2$. So the map $\ker(\varphi_{H^i}-p)_n\to \ker(\varphi_{H^i_n}-p)$ is an al. isomorphism uniformly in n. Meanwhile, use the same argument as in Lemma \ref{lemmatotorsion},
    
    the map $\ker(\varphi_{\cM}-p)/p^n\to \ker(\varphi_{\cM_n}-p)$ (where $\varphi_{\cM_n}-p$ is the map associated to $R^if_{\crys*}\cM^{cr}(G[p^n])$) is an al. isomorphism uniformly in $n$. Therefore, the induced maps $\ker(\varphi_{H_n^i}-p)\to \ker(\varphi_{\cM_n}-p)$ are compatible and al. isomorphic uniformly in $n$. Push it forward along $u: S_{\CRYS,\SYN}\to S_{\SYN}$, we obtain almost isomorphisms uniformly in $n$:
    \begin{equation}\label{formulacomposition}\begin{aligned}
       w_n: H^i_n\stackrel{\cong^a}\to \ker(u_*(\varphi_{H^i_n}-p))=&u_*\ker(\varphi_{H^i_n}-p)\stackrel{\cong^a}\to \\
        &u_*\ker(\varphi_{\cM_n}-p)=\ker(u_*(\varphi_{\cM_n}-p)).
        \end{aligned}\end{equation}
   
    We still use $\ker(u_*(\varphi_{\cM_n}-p))$ to denote its pushforward to small syntomic site. Since the map $R^if_{\syn*}\cM(G[p^n])\to u_*R^if_{\crys*}\cM^{cr}(G[p^n])$ is an al. isomorphism uniformly in $n$ (Lemma \ref{lemmasomecollection} (2)), by proposition \ref{propmain1}, we get al. isomorphisms of sheaves on $\syn(S)$: $r_n: R^if_{\syn*}(G[p^n]) \stackrel{\cong^a}\to\ker(u_*(\varphi_{\cM_n}-p))$  uniformly in $n$. By Lemma \ref{findisogeny}, there exists compatible maps $l_n: \ker(u_*(\varphi_{\cM_n}-p))\to R^if_{\syn*}G[p^n]$ such that compositions of $r_n,l_n$ in both direction are $p^N$. Compose \eqref{formulacomposition} with $l_n$, we get compatible maps $P_n': H^i[p^n]\to R^if_{\syn*}G[p^n]$ which is an al. isomorphism uniformly in $n$. Apply Lemma \ref{findisogeny} again, we can find compatible maps 
    \begin{equation}\label{formulafinal}
        P'_n: H^i[p^n]\rightleftarrows R^if_{\syn*}G[p^n] : Q'_n.
    \end{equation} 
    
    By \cite[Cor. 1.4]{bragg2021representabilitycohomologyfiniteflat}, each sheaf $R^if_{\fppf*}G[p^n]$ is represented by an affine group scheme of finite type over $S$ (only require $f$ is projective not necessarily smooth).  Every group scheme of finite type over $k$ is syntomic (\cite[Prop. 27.26]{gortz2023algebraic}) and $v_*R^if_{\fppf*}G[p^n]=R^if_{\syn*}G[p^n]$ (Lemma \ref{fppfsyn}). So both sides of \eqref{formulafinal} are representable in $\syn(S)$ and we get corresponding maps between the schemes. In particular, we get maps between fppf sheaves
    $$P_n: H^i[p^n]\rightleftarrows R^if_{\fppf*}G[p^n] : Q_n$$
    such that $Q'_n\circ P'_n=p^N$, $P'_n\circ Q'_n=p^N$.
   In conclusion, $R^if_{\fppf*}G$ is isogeny to the $p$-divisible group $H^i$ and 
   $$\cM^{cr}(H^i)_{\bQ}\cong R^if_{\crys*}\cM^{cr}(G)_{\bQ,[0,1]}$$
   as desired.
\end{proof}

\begin{remark}
By Theorem \ref{theoremmain}, $R^if_{\fppf*}G$ satisfies the condition of Definition \ref{defpdivisible} and there is an isomorphism between $F$-isocrystals 
$$\cM^{cr}(\Div(R^if_{\fppf*}G))_\bQ\cong R^if_{\crys*}\cM^{cr}(G)_{\bQ,[0,1]}.$$
This isomorphism, after multiplying with a suitable power of $p$ (to make $r,l$ and $P,Q$ be inverse of each other in the isogeny category), is independent of the choice of $H^i$. This follows from a diagram chasing.

\end{remark}

\begin{corollary}
     We have an isomorphism
      $$(\varprojlim_n H^0_{\fppf}(S,R^if_{\fppf*}G[p^n]))_{\bQ}\cong H^0_{\crys}(S/W,R^if_{\crys*}\cM^{cr}(G))_{\bQ}^{\varphi=p}$$
 \end{corollary}

 \begin{proof}
     Let $H^i$ be a p-divisible group isogeny to $R^if_{\syn*}G$. We have the following isomorphisms:
     $$\begin{aligned}
         (\varprojlim_nH^0_{\fppf}(S,R^if_{\syn*}(G[p^n])))_{\bQ}& \cong  (\varprojlim_nH^0_{\fppf}(S,H^i[p^n]))_{\bQ}\\
         &\cong \Hom_{p\text{-div}_S}(\bQ_p/\bZ_p, H^i)_{\bQ}\\
         &\cong \Hom_{\Fisoc}(\cO(-1),\cM(H^i))\\
         &\cong H^0_{\crys}(S/W,R^if_{\crys*}\cM^{cr}(G)_{[0,1]})_{\bQ}^{\varphi=p}\\
         &\cong H^0_{\crys}(S/W,R^if_{\crys*}\cM^{cr}(G))_{\bQ}^{\varphi=p}
     \end{aligned}$$
     where the last isomorphism follows from the proof of theorem \ref{theoremmain}.
 \end{proof}

\bibliographystyle{amsplain}
\bibliography{fppfcrys}

\providecommand{\bysame}{\leavevmode\hbox to3em{\hrulefill}\thinspace}
\providecommand{\MR}{\relax\ifhmode\unskip\space\fi MR }
\providecommand{\MRhref}[2]{%
  \href{http://www.ams.org/mathscinet-getitem?mr=#1}{#2}
}
\providecommand{\href}[2]{#2}
\begin{thebibliography}{10}

\bibitem{ArtinMazurformal}
M.~Artin and B.~Mazur, \emph{Formal groups arising from algebraic varieties}, Annales scientifiques de l'\'Ecole Normale Sup\'erieure \textbf{Ser. 4, 10} (1977), no.~1, 87--131 (en).

\bibitem{Bauer1992}
Werner Bauer, \emph{On the conjecture of {B}irch and {S}winnerton-{D}yer for abelian varieties over function fields in characteristic p $>$ 0.}, Inventiones mathematicae \textbf{108} (1992), no.~2, 263--288.

\bibitem{berthelot2006theorie}
Pierre Berthelot, Lawrence Breen, and William Messing, \emph{Th{\'e}orie de {D}ieudonn{\'e} {C}ristalline {II}}, vol. 930, Springer, 2006.

\bibitem{BerthelotOgus1978}
Pierre Berthelot and Arthur Ogus, \emph{Notes on {C}rystalline {C}ohomology}, Princeton University Press, Princeton, 1978.

\bibitem{bragg2021representabilitycohomologyfiniteflat}
Daniel Bragg and Martin Olsson, \emph{Representability of cohomology of finite flat abelian group schemes},  (2021), arXiv 2107.11492 preprint at \url{https://arxiv.org/abs/2107.11492}.

\bibitem{de2000purity}
A~De~Jong and Frans Oort, \emph{Purity of the stratification by {N}ewton polygons}, Journal of the American Mathematical Society \textbf{13} (2000), no.~1, 209--241.

\bibitem{de1999dieudonne}
Aise~Johan De~Jong and William Messing, \emph{Crystalline {Dieudonn\'e} theory over excellent schemes}, Bulletin de la Soci\'et\'e Math\'ematique de France \textbf{127} (1999), no.~2, 333--348 (en). \MR{2001b:14075}

\bibitem{de1995crystalline}
Arthur~J de~Jong, \emph{Crystalline {D}ieudonn{\'e} module theory via formal and rigid geometry}, Publications Math{\'e}matiques de l'IH{\'E}S \textbf{82} (1995), 5--96.

\bibitem{drinfeld2018theorem}
Vladimir Drinfeld, \emph{On a theorem of {S}cholze-{W}einstein}, arXiv preprint arXiv:1810.04292 (2018).

\bibitem{gortz2023algebraic}
Ulrich G{\"o}rtz and Torsten Wedhorn, \emph{Algebraic {G}eometry {II}: {C}ohomology of {S}chemes}, 1 ed., Springer Studium Mathematik - Master, Springer Spektrum, Wiesbaden, 2023.

\bibitem{goertz-wedhorn2020}
Ulrich Görtz and Torsten Wedhorn, \emph{Algebraic {G}eometry {I}: {S}chemes -- {W}ith {E}xamples and {E}xercises}, 2 ed., Springer Studium Mathematik - Master, Springer Spektrum, Wiesbaden, 2020.

\bibitem{kato1988expose}
Kazuya Kato, \emph{Expos{\'e} vi: Semi-stable reduction and p-adic etale cohomology}, P{\'e}riodes p-adiques-S{\'e}minaire de Bures (1988), no.~223, 269--293.

\bibitem{katzslope}
Nicholas~M. Katz, \emph{Slope filtration of $f$-crystals}, Journ\'ees de G\'eom\'etrie Alg\'ebrique de Rennes - (Juillet 1978) (I) : Groupe formels, repr\'esentations galoisiennes et cohomologie des vari\'et\'es de caract\'eristique positive, Ast\'erisque, no.~63, Soci\'et\'e math\'ematique de France, 1979, pp.~113--163 (en). \MR{563463}

\bibitem{matthiassyntomic}
Matthias Kummerer, \emph{Syntomic cohomology},  (2014), \url{https://www.mathi.uni-heidelberg.de/~otmar/diplom/kuemmerer.pdf}.

\bibitem{li2024ptorsionsgeometricbrauergroups}
Zhenghui Li and Yanshuai Qin, \emph{On p-torsions of geometric {B}rauer groups}, 2024, \url{https://arxiv.org/abs/2406.19518}.

\bibitem{Ma-Ro}
B~Mazur and L~Roberts, \emph{Local {E}uler characteristics}, Inventiones mathematicae \textbf{9} (1970), 201--234.

\bibitem{messing1972crystals}
William Messing, \emph{The {C}rystals {A}ssociated to {B}arsotti-{T}ate groups: {W}ith {A}pplications to {A}belian {S}chemes}, Lecture Notes in Mathematics, vol. 264, Springer-Verlag, Berlin, Heidelberg, 1972.

\bibitem{Mil3}
James~S Milne, \emph{Etale cohomology (pms-33)}, Princeton university press, 1980.

\bibitem{Mor}
M.~Morrow, \emph{A {V}ariational {T}ate {C}onjecture in crystalline cohomology}, Journal of the European Mathematical Society \textbf{21} (2019), no.~11, 3467--3511.

\bibitem{NO}
N.~Nygaard and A.~Ogus, \emph{{T}ate's conjecture for {K}3 surfaces of finite height}, Ann. of Math. \textbf{122} (1985), 461--507.

\bibitem{oda1969first}
T.~Oda, \emph{The first de {R}ham cohomology group and {D}ieudonn{\'e} modules}, Annales scientifiques de l'{\'E}cole Normale Sup{\'e}rieure \textbf{2} (1969), no.~1, 63--135.

\bibitem{Pal}
Ambrus P{\'a}l, \emph{The $ p $-adic monodromy group of abelian varieties over global function fields of characteristic $ p$}, Documenta Mathematica \textbf{27} (2022), 1509--1579.

\bibitem{tatedualitypositivedimension}
Zev Rosengarten, \emph{Tate duality in positive dimension over function fields}, 2023, arXiv 1805.00522v9, \url{https://arxiv.org/abs/1805.00522}.

\bibitem{scholze2013moduli}
Peter Scholze and Jared Weinstein, \emph{Moduli of p-divisible groups}, Cambridge Journal of Mathematics \textbf{1} (2013), no.~2, 145--237.

\bibitem{stacks-project}
The {Stacks project authors}, \emph{The {S}tacks project}, \url{https://stacks.math.columbia.edu}, 2024.

\bibitem{trihan2018comparisontheoremsemiabelianschemes}
Fabien Trihan and David Vauclair, \emph{A comparison theorem for semi-abelian schemes over a smooth curve}, 2018, arXiv:1505.02942v2, \url{https://arxiv.org/abs/1505.02942}.

\end{thebibliography}
\end{document}